\documentclass{amsart}

\usepackage{latexsym,amssymb,amsmath,amscd,amsxtra,amsfonts,amsthm}
\usepackage{fancyhdr,array,dsfont}

\newtheorem{theorem}{Theorem}[section]
\newtheorem{lemma}[theorem]{Lemma}
\newtheorem{proposition}[theorem]{Proposition}

\newtheorem{corollary}[theorem]{Corollary}

\newcommand{\beq}{\begin{eqnarray*}}
\newcommand{\eeq}{\end{eqnarray*}}
\newcommand{\beqn}{\begin{eqnarray}}
\newcommand{\eeqn}{\end{eqnarray}}

\numberwithin{equation}{section}

\begin{document}

%%%%% To ease editing, for IMPAN journals add:

%\baselineskip=17pt

%%%%%%%%%%%

%% In the running head, replace first names by initials
%% and give an abbreviation of the title.

\title[A link between the log-Sobolev inequality and Lyapunov condition]{A link between the log-Sobolev inequality and Lyapunov condition}

\author[Y. LIU]{Yuan LIU}
\address{Yuan LIU, Institute of Applied Mathematics, Academy of Mathematics and Systems Science,
Chinese Academy of Sciences, Beijing 100190, China}
\email{liuyuan@amss.ac.cn}

\date{\today}

\begin{abstract}
We give an alternative look at the log-Sobolev inequality (LSI in short) for log-concave measures by semigroup tools. The similar idea yields a heat flow proof of LSI under some quadratic Lyapunov condition for symmetric diffusions on Riemannian manifolds provided the Bakry-Emery's curvature is bounded from below. Let's mention that, the general $\phi$-Lyapunov conditions were introduced by Cattiaux-Guillin-Wang-Wu \cite{CGWW} to study functional inequalities, and the above result on LSI was first proved subject to $\phi(\cdot)=d^2(\cdot, x_0)$ by Cattiaux-Guillin-Wu \cite{CGW} through a combination of detective $L^2$ transportation-information inequality $\mathrm{W\hspace*{-0.5mm}_2I}$ and the HWI inequality of Otto-Villani.

Next, we assert a converse implication that the Lyapunov condition can be derived from LSI, which means their equivalence in the above setting.
\end{abstract}



\maketitle

\allowdisplaybreaks

\section{Introduction}
 \label{Intro}
 \setcounter{equation}{0}
Our aim of this paper is to give a direct proof of the log-Sobolev inequality (LSI for short) for symmetric diffusions under the Lyapunov condition and curvature condition, and the converse implication will be investigated further.

In the sequel, denote by $E$ a connected complete Riemannian manifold of finite dimension, $d$ the geodesic distance, $\mathrm{d}x$ the volume measure, $\mu(\mathrm{d}x) = e^{-V(x)}\mathrm{d}x$ a probability measure with $V\in C^2(E)$, $\mathrm{L}=\Delta - \nabla V\cdot \nabla$ the $\mu$-symmetric diffusion operator with domain $\mathcal{D}(\mathrm{L})$ in $L^2(\mu)$, $P_t = e^{t\mathrm{L}}$ the  semigroup, $\Gamma(f,g) = \nabla f\cdot \nabla g$ the carr\'{e} du champ operator, and $\mathcal{E}(f,g) = \int \Gamma(f,g) \mathrm{d}\mu$ the Dirichlet form with domain $\mathcal{D}(\mathcal{E})$ in $L^2(\mu)$. It is known that the integration by parts formula reads
   \[ \mathcal{E}(f,g) = -\int f \mathrm{L}g  \;\mathrm{d}\mu, \ \ \forall f\in \mathcal{D}(\mathcal{E}), g\in \mathcal{D}(\mathrm{L}), \]
and $P_t$ is $L^2$-ergodic, i.e.
   \[ ||P_tf-\mu f||_{L^2(\mu)} \to 0 \ \textrm{ as } t\to \infty, \ \forall f\in L^2(\mu). \]
We refer to Bakry-Gentil-Ledoux \cite{BGL} for a detailed presentation of the fundamentals. For simplicity, write $\mu f^2 =\int f^2 \mathrm{d}\mu$, $P_t f^2 = P_t(f^2)$, and $\mathcal{E}[f]=\mathcal{E}(f,f)$.

A (tight) LSI means there exists a constant $C>0$ such that for any $f\in \mathcal{D}(\mathcal{E})$
   \beqn
     \mathrm{Ent}(f^2) := \mu(f^2\log f^2) - \mu f^2\log\mu f^2 \leqslant C\mathcal{E}(f,f). \label{eqLSI}
   \eeqn
There have been several classical proofs of LSI for log-concave measures, such as \cite{BE, Bobkov-Ledoux, CHL, Ledoux-01}. We would like to revisit this case in another simple viewpoint.

\begin{proposition} \label{thmLSILogConcave}
Suppose $\mu$ is log-concave in $\mathbb{R}^n$ such that $\mathrm{Hess}(V)\geqslant c\mathrm{Id}$ with $c>0$. Then the LSI (\ref{eqLSI}) holds for $C=2/c$.
\end{proposition}

Set $\Phi(t) = 2\mathcal{E}[\sqrt{P_t f}] - c\mathrm{Ent}(P_tf)$ for $f>0$. Heuristically, due to $P_t f\to\mu f$ in $L^2$-norm, we have $\Phi(\infty) =  0$ and then $\Phi(0) \geqslant 0$ if there holds the monotonicity
   \beqn
     \frac{\mathrm{d}\Phi}{\mathrm{d}t} \leqslant 0. \label{eqMono}
   \eeqn
Actually, using the integration by parts formula and next two equalities for $\varphi= P_tf$
  \[ \frac{\mathrm{\partial}}{\mathrm{\partial}t} \varphi = \mathrm{L}\varphi,
   \ \ \ \nabla\mathrm{L}\varphi = \mathrm{L}\nabla \varphi - \mathrm{Hess}(V)\nabla\varphi, \]
we can calculate $\frac{\mathrm{d}\Phi}{\mathrm{d}t}$ to imply (\ref{eqMono}) automatically.

For general cases, some practical conditions have been presented to derive LSI. In particular, Cattiaux-Guillin-Wang-Wu \cite{CGWW} introduced the $\phi$-Lyapunov conditions to study super Poincar\'{e} inequalities and LSI. Cattiaux-Guillin-Wu \cite{CGW} mainly discussed the ``quadratic" Lyapunov condition as $\phi(\cdot)=d^2(\cdot, x_0)$ to derive the Talagrand's inequality $\mathrm{W\hspace*{-0.5mm}_2H}$, and also LSI provided that the Bakry-Emery's curvature has a lower bound (actually, they gave more than that when the curvature had a decay rate and the Lyapunov condition became stronger correspondingly). So our first aim comes out of here. Recall that the proof of LSI in \cite{CGW} relied on a combination of detective $L^2$ transportation-information inequality $\mathrm{W\hspace*{-0.5mm}_2I}$ and the HWI inequality in Otto-Villani \cite{Otto-Villani}. We want to give a direct proof of this fact via the similar idea of monotonicity on heat flow as for Proposition \ref{thmLSILogConcave}.

With a little relaxation, say $W>0$ is a Lyapunov function if $W^{-1}$ is locally bounded and there exist two constants $c>0, b\geqslant 0$ and some $x_0\in E$ such that in the sense of distribution
  \beqn
     \mathrm{L}W \leqslant (-cd^2(x,x_0) + b) W. \label{eqLya}
  \eeqn
Note that we request $W>0$ other than $W\geqslant 1$ in \cite{CGW}, but if $W^{-1}$ is locally bounded, the technique of Bakry-Barthe-Cattiaux-Guillin \cite[Theorem 1.4]{BBCG} still works, which plays an important role in studying various functional inequalities via the Lyapunov type conditions. We prove that

\begin{theorem} \label{thmLya}
Suppose $\mathrm{Ric}+\mathrm{Hess}(V)\geqslant -K \mathrm{Id}$ with $K\in \mathbb{R}^+$. Then the LSI (\ref{eqLSI}) holds under the Lyapunov condition (\ref{eqLya}).
\end{theorem}

We further investigate how to derive the Lyapunov condition from LSI.
\begin{theorem} \label{thmLSILya}
If the LSI (\ref{eqLSI}) holds, there exists $W>0$ satisfying (\ref{eqLya}).
\end{theorem}

As consequence, it follows
\begin{corollary} \label{corLSILya}
Suppose $\mathrm{Ric}+\mathrm{Hess}(V)\geqslant -K \mathrm{Id}$ with constant $K\geqslant 0$. Then the LSI (\ref{eqLSI}) is equivalent to the Lyapunov condition (\ref{eqLya}).
\end{corollary}

Since the Poincar\'{e} inequality will serve as a basic tool in our proofs, it is necessary to take some related literature into account. Cattiaux-Guillin-Zitt \cite{CGZ} sought out the links around the Poincar\'{e} inequality, exponential contraction of diffusion semigroup, exponential integrability of hitting time, and (weak) Lyapunov condition as
   \[ \mathrm{L}W \leqslant -\lambda W \ \textrm{on}\ \bar{U}^c, \ \ \textrm{with}\ W\geqslant 1, \lambda>0, U \ \textrm{open connected and bounded}, \]
which made a continuous-time analogy with those characterizations for irreducible and aperiodic Markov chains on countable state spaces. Bakry-Cattiaux-Guillin \cite{BCG} discussed also the relations between the Poincar\'{e} inequality and certain Lyapunov condition associated with a closed Petite set. Note that our argument for Corollary \ref{corLSILya} can not be applied to the above quoted results.

Next two sections will be devoted to the proofs of Theorem \ref{thmLya}-\ref{thmLSILya} respectively.

\bigskip
\section{Proof of Theorem \ref{thmLya}}
\label{Lyapunov}
 \setcounter{equation}{0}

First of all, we show that the Lyapunov condition (\ref{eqLya}) acts the same as in \cite{CGW}.

\begin{lemma}(\cite[Theorem 1.4]{BBCG}) \label{lemLya}
The Poincar\'{e} inequality holds under the Lyapunov condition (\ref{eqLya}). Moreover, for any $h\in \mathcal{D}(\mathcal{E})$
   \beqn
      \int h^2(x)d^2(x,x_0)\mathrm{d}\mu(x) \leqslant \frac{1}{c} \int |\nabla h|^2 \mathrm{d}\mu + \frac{b}{c}\int h^2 \mathrm{d}\mu. \label{eqTransLya}
   \eeqn
\end{lemma}
\begin{proof}
It follows from (\ref{eqLya}) that
  \[ \mathrm{L} W \leqslant \left[ -c+(c+b)\mathbf{1}_{B_1(0)} \right] W, \]
which implies the Poincar\'{e} inequality by the argument of \cite[Page 64]{BBCG}. Moreover,
  \beq
    \int h^2(x)d^2(x,x_0)\mathrm{d}\mu(x)
    &=& \frac{1}{c} \int h^2(cd^2(x,x_0)-b) \mathrm{d}\mu + \frac{b}{c}\int h^2 \mathrm{d}\mu \\
    &\leqslant& \frac{1}{c} \int \frac{-\mathrm{L}W}{W} h^2 \mathrm{d}\mu + \frac{b}{c}\int h^2 \mathrm{d}\mu \\
    &=& \frac{1}{c} \int \nabla W \cdot \nabla \frac{h^2}{W}  \mathrm{d}\mu + \frac{b}{c}\int h^2 \mathrm{d}\mu \\
    &=& \frac{1}{c} \int |\nabla h|^2 - \left| \nabla h - \frac{h}{W}\nabla W \right|^2  \mathrm{d}\mu + \frac{b}{c}\int h^2 \mathrm{d}\mu \\
    &\leqslant& \frac{1}{c} \int |\nabla h|^2 \mathrm{d}\mu + \frac{b}{c}\int h^2 \mathrm{d}\mu.
  \eeq
Note that here is no need to assume the integrability of $d^2(x,x_0)$ or $\frac{-\mathrm{L}W}{W}$ for $\mu$, since we can take an approximation sequence in $C_{\textrm{c}}^\infty(E)$ for given $h$.
\end{proof}

\begin{lemma} \label{lemRic}
The curvature condition $\mathrm{Ric}+ \mathrm{Hess}(V)\geqslant -K$ with $K\in \mathbb{R}$ implies
\begin{enumerate}
\item For any $f\in \mathcal{B}_b^+$, $x,y\in E$ and $t>0$
   \beqn
      (P_tf)^2(x)\leqslant P_tf^2(y)\exp\left( \frac{Kd^2(x,y)}{1-e^{-2Kt}}\right). \label{eqPtbound}
   \eeqn
\item For any $f\in C_{\mathrm{b}}^1$ and $t>0$
   \beqn
      |\nabla P_tf| \leqslant e^{Kt} P_t|\nabla f|. \label{eqPtGradient}
   \eeqn
\end{enumerate}
\end{lemma}
\begin{proof}
Refer to \cite[Section 5.5-5.6]{BGL} or Wang \cite[Theorem 2.3.3]{Wang-book}.
\end{proof}

\begin{corollary} \label{corPtbound}
If $\mathrm{Ric}+ \mathrm{Hess}(V)\geqslant -K$ with $K\in \mathbb{R}^+$, set $\mu_0 = \mu e^{-2Kd^2(x_0,\cdot)}$, then
   \beqn
      \frac{(P_tf)^2(x)}{\mu f^2}
      \leqslant   \mu_0^{-\frac{1}{1-e^{-2Kt}}}\exp\left(\frac{2K}{1-e^{-2Kt}}d^2(x_0,x)\right). \label{eqPtx0Upper}
   \eeqn
\end{corollary}
\begin{proof}
Denote $\delta(t) = 1-e^{-2Kt}\leqslant 1$. Using Lemma \ref{lemRic} yields
  \beq
      \left( P_tf^2(y)\right)^{\delta(t)}
      &\geqslant& (P_tf)^{2\delta(t)}(x)\exp\left( -Kd^2(x,y)\right)\\
      &\geqslant& (P_tf)^{2\delta(t)}(x)\exp\left( -2Kd^2(x_0,x)-2Kd^2(x_0,y)\right).
   \eeq
Integrating in $y$ on both sides gives
   \[ \mu\left( \left(P_tf^2\right)^{\delta(t)} \right) \geqslant (P_tf)^{2\delta(t)}(x)\exp\left( -2Kd^2(x_0,x)\right)\mu_0,\]
which implies (\ref{eqPtx0Upper}) by the H\"{o}lder inequality $\mu\left( \left(P_tf^2\right)^{\delta(t)} \right) \leqslant \left(\mu(P_tf^2)\right)^{\delta(t)}$.
\end{proof}

Now we prove Theorem \ref{thmLya}.
\begin{proof}
The strategy contains three steps. Assume $f$ is a bounded smooth function.

{\bf Step 1}. Abbreviate $\varphi = P_t f$, we introduce
  \[ \mathrm{Ent}^*(\varphi^2) := \int \varphi^2 \log\frac{\varphi^2}{\mu f^2} \mathrm{d}\mu, \]
and
  \[ \Psi(t) :=  \mathcal{E}[\varphi] + A\mu (\varphi - \mu \varphi)^2
         - \eta\mathrm{Ent}^*(\varphi^2), \]
where $A$ and $\eta$ are two positive constants which will be decided below.

Derivative calculations give respectively
  \beqn
      \frac{\mathrm{d}}{\mathrm{d}t}\mathcal{E}[\varphi] &=&
         - 2\mu |\nabla \nabla\varphi|^2 - 2\mu \left[\left(\mathrm{Ric}+\mathrm{Hess}(V)\right)(\nabla\varphi, \nabla\varphi)\right], \label{eqDe11}\\
      A\frac{\mathrm{d}}{\mathrm{d}t} \mu (\varphi - \mu \varphi)^2 &=&
         -2A\mu (|\nabla \varphi|^2),\label{eqDe12}\\
      -\eta\frac{\mathrm{d}}{\mathrm{d}t} \mathrm{Ent}^*(\varphi^2) &=&
         2\eta\mu \left(|\nabla \varphi|^2 \log\frac{\varphi^2}{\mu f^2}\right) + 6\eta\mu |\nabla\varphi|^2. \label{eqDe13}
   \eeqn
Here $\nabla\nabla \varphi$ denotes the Hessian of $\varphi$. Due to the curvature condition, (\ref{eqDe11}) is less than $- 2\mu |\nabla \nabla\varphi|^2 + 2K\mu|\nabla\varphi|^2$. To estimate (\ref{eqDe13}), using (\ref{eqPtx0Upper}) gives
   \[ \mu \left(|\nabla \varphi|^2 \log\frac{\varphi^2}{\mu f^2}\right) \leqslant
          \frac{2K}{1-e^{-2Kt}} \mu \left(|\nabla \varphi|^2 d^2(x,x_0)\right) -\frac{\log \mu_0}{1-e^{-2Kt}} \mu |\nabla \varphi|^2.\]

We fix $t_0=1$ (or any positive number) and set $\eta=c \frac{1-e^{-2Kt_0}}{2K}$. Applying (\ref{eqTransLya}) to the above inequality for $h^2=|\nabla \varphi|^2$ yields for all $t\geqslant t_0$
   \beqn
      \hspace*{1cm} 2\eta \mu \left(|\nabla \varphi|^2 \log\frac{\varphi^2}{\mu f^2}\right) \leqslant
        2\mu |\nabla \nabla\varphi|^2 + 2b\mu|\nabla \varphi|^2 - \frac{c\log \mu_0}{K}\mu |\nabla \varphi|^2. \label{eqApplyLya}
   \eeqn

Set $A=K+b-\frac{c\log \mu_0}{2K}+3\eta$ (note that $\mu_0\leqslant 1$), we obtain by combining (\ref{eqDe11}-\ref{eqDe13}) with (\ref{eqApplyLya})
  \[ \frac{\mathrm{d}}{\mathrm{d}t}\Psi(t) \leqslant 0, \ \ \forall t\geqslant t_0, \]
which implies
  \[ \Psi(t_0) \geqslant \Psi(t) \geqslant \Psi(\infty)\geqslant 0, \ \ \forall t\geqslant t_0, \]
namely
  \[ \eta\mathrm{Ent}^*(\varphi^2) \leqslant \mathcal{E}[\varphi] + A\mu (\varphi - \mu \varphi)^2, \ \ \forall t\geqslant t_0. \]

{\bf Step 2}. For $0<t<t_0$, it is invalid to prove $\frac{\mathrm{d}}{\mathrm{d}t}\Psi(t) \leqslant 0$ in the above manner. Nevertheless, we turn to comparing $\mathrm{Ent}(P_{t_0}f^2)$ with $\mathrm{Ent}^*\left((P_{t_0}f)^2\right)$ directly (and thus, it is allowed to alter the definition of $\Psi$ for $0<t<t_0$). Define
  \[ \Theta_1(t) = \int \left( P_{t}f^2 - (P_{t}f)^2 \right) \log \frac{(P_{t}f)^2}{\mu f^2}\mathrm{d}\mu, \ \ \
     \Theta_2(t) = \int P_tf^2 \log\frac{P_tf^2}{(P_t f)^2} \mathrm{d}\mu, \]
which satisfy
  \[ \mathrm{Ent}(P_{t_0}f^2) - \mathrm{Ent}^*\left((P_{t_0}f)^2\right) = \Theta_1(t_0) + \Theta_2(t_0). \]

Firstly, it follows from (\ref{eqPtx0Upper})
   \beq
      &&\Theta_1(t_0)\\
      &\leqslant& \frac{2K}{1-e^{-2Kt_0}} \mu {\Big[} \left( P_{t_0}f^2- (P_{t_0}f)^2 \right)d^2(x,x_0) {\Big]} -\frac{\log \mu_0}{1-e^{-2Kt_0}} \mu \left(P_{t_0}f^2 - (P_{t_0}f)^2 \right)\\
      &\leqslant& \frac{2K}{1-e^{-2Kt_0}} \mu \left(P_{t_0}f^2 \cdot d^2(x,x_0)\right) -\frac{\log \mu_0}{1-e^{-2Kt_0}}\mu\left(f^2-(\mu f)^2\right).
   \eeq
Using (\ref{eqTransLya}), (\ref{eqPtGradient}) and the H\"{o}lder inequality yields
   \beq
      \mu \left(P_{t_0}f^2 \cdot d^2(x,x_0)\right)
      &\leqslant& c^{-1}\mu \left|\nabla \sqrt{P_{t_0}f^2}\right|^2 + bc^{-1} \mu f^2\\
      &=& (4c)^{-1}\mu \frac{\left|\nabla P_{t_0}f^2\right|^2}{P_{t_0}f^2} + bc^{-1} \mu f^2\\
      &\leqslant& (4c)^{-1}e^{2Kt_0} \mu \frac{\left(P_{t_0}\left|\nabla f^2\right|\right)^2}{P_{t_0}f^2} + bc^{-1} \mu f^2\\
      &=& c^{-1}e^{2Kt_0} \mu \frac{\left(P_{t_0}\left|f\nabla f\right|\right)^2}{P_{t_0}f^2} + bc^{-1} \mu f^2\\
      &\leqslant& c^{-1}e^{2Kt_0} \mu |\nabla f|^2 + bc^{-1} \mu f^2.
   \eeq
Combining the above two estimates gives
   \beqn
      \Theta_1(t_0)
      \leqslant C_1 \mathcal{E}[f] + C_2\mu f^2 + C_3 \mu (f-\mu f)^2, \label{eqCompare01}
   \eeqn
where $C_1 = \frac{2Ke^{2Kt_0}}{c(1-e^{-2Kt_0})}$, $C_2 = \frac{2bK}{c(1-e^{-2Kt_0})}$ and $C_3 = -\frac{\log \mu_0}{1-e^{-2Kt_0}}$.

Secondly, due to $\Theta_2(0) = 0$, there is an integral representation
   \beq
     \Theta_2(t_0) &=& \int_0^{t_0} \Theta_2'(t) \mathrm{d}t \\
     &=& \int_0^{t_0} \int \mathrm{L}P_tf^2 \log\frac{P_tf^2}{(P_t f)^2} + \mathrm{L}P_tf^2 - 2\frac{P_tf^2}{P_t f} \mathrm{L}P_tf \;\mathrm{d}\mu\mathrm{d}t \\
     &=& \int_0^{t_0} \int \mathrm{L}P_tf^2 \log\frac{P_tf^2}{(P_t f)^2} - 2\frac{P_tf^2}{P_t f} \mathrm{L}P_tf \;\mathrm{d}\mu\mathrm{d}t\\
     &=& \int_0^{t_0} \int -\frac{\left|\nabla P_tf^2\right|^2}{P_t f^2} + 4\frac{\nabla P_tf^2 \cdot \nabla P_t f}{P_t f} - 2\frac{P_tf^2\left|\nabla P_tf\right|^2}{(P_t f)^2} \;\mathrm{d}\mu\mathrm{d}t\\
     &\leqslant& \int_0^{t_0} \int \frac{\left|\nabla P_tf^2\right|^2}{P_t f^2} - 2\left( \frac{\left|\nabla P_tf^2\right|}{\sqrt{P_t f^2}} - \frac{\sqrt{P_t f^2}\left|\nabla P_tf\right|}{P_tf} \right) ^2 \;\mathrm{d}\mu\mathrm{d}t,
   \eeq
which implies through (\ref{eqPtGradient}) and the H\"{o}lder inequality
   \beqn
      \Theta_2(t_0) &\leqslant& \int_0^{t_0} \mu\frac{\left|\nabla P_tf^2\right|^2}{P_t f^2}\; \mathrm{d}t
          \ \leqslant\ \int_0^{t_0} 4e^{2Kt} \mu \frac{\left(P_t|f \nabla f| \right)^2}{P_t f^2}\; \mathrm{d}t \nonumber\\
          &\leqslant& \int_0^{t_0} 4e^{2Kt} \mu|\nabla f|^2\; \mathrm{d}t \ =:\ C_4 \mathcal{E}[f], \label{eqLambdaBound}
   \eeqn
where $C_4 = \frac{2(e^{2Kt_0} -1)}{K}$.

Combining (\ref{eqCompare01}) with (\ref{eqLambdaBound}) gives
   \[ \mathrm{Ent}(P_{t_0}f^2) - \mathrm{Ent}^*\left((P_{t_0}f)^2\right)
        \leqslant (C_1+C_4)\mathcal{E}[f] + C_2\mu f^2 + C_3 \mu (f-\mu f)^2. \]
Recall the last inequality in Step 1 together with the monotonicity of $\mathcal{E}[\varphi]$ in $t$ (see \cite[Proposition 3.1.6]{BGL}), we obtain
   \[ \eta \mathrm{Ent}(P_{t_0}f^2) \leqslant \left[1+\eta(C_1+C_4)\right]\mathcal{E}[f] + \eta C_2\mu f^2 + (A+\eta  C_3) \mu (f-\mu f)^2. \]

{\bf Step 3}. Now it is feasible to employ a new form of $\Psi$ for $0<t<t_0$ like
  \[ (1+t_0 - t)\left(A_1\mathcal{E}[f] + A_2\mu f^2 + A_3\mu (f - \mu f)^2\right) - \eta\mathrm{Ent}(P_t  f^2 )\]
to show its monotonicity further. In an equivalent and quick way, it is enough to use the same argument as (\ref{eqLambdaBound}) to get
   \beq
       \mathrm{Ent}(f^2) - \mathrm{Ent}(P_{t_0}f^2)
       &=& \int_0^{t_0} -\frac{\mathrm{d}}{\mathrm{d}t} \mathrm{Ent}(P_tf^2) \mathrm{d}t \\
       &=& \int_0^{t_0}\int -\mathrm{L}P_tf^2 \log\frac{P_t f^2}{\mu f^2} -\mathrm{L}P_tf^2 \;\mathrm{d}\mu \mathrm{d}t\\
       &=& \int_0^{t_0} \mu \frac{\left|\nabla P_tf^2\right|^2}{P_tf^2} \mathrm{d}t \ \leqslant\ C_4 \mathcal{E}[f].
   \eeq
Recall the last inequality in Step 2, it follows
   \[ \eta \mathrm{Ent}(f^2) \leqslant \left[1+\eta(C_1+2C_4)\right]\mathcal{E}[f] + \eta C_2\mu f^2 + (A+\eta C_3)\mu (f-\mu f)^2. \]

Denote by $\lambda_\mu$ the spectral gap, applying the Poincar\'{e} inequality to the above estimate yields
   \beqn
     \ \ \ \ \ \ \eta \mathrm{Ent}(f^2) \leqslant \left[1+\eta(C_1+2C_4) + \lambda_{\mu}^{-1}(A+\eta C_3)\right]\mathcal{E}[f] + \eta C_2\mu f^2. \label{eqPreLast}
   \eeqn
Substituting $f$ in (\ref{eqPreLast}) to $f-\mu f$, we can bound $\eta\mathrm{Ent}((f-\mu f)^2)$ by some linear combination of $\mathcal{E}[f]$ and $\mu (f-\mu f)^2$.

On the other hand, applying the Rothaus's lemma in \cite{Rothaus} i.e. $\mathrm{Ent}((f+a)^2) \leqslant \mathrm{Ent}(f^2) + 2\mu f^2$ for any $a\in \mathbb{R}$ yields
   \beqn
     \mathrm{Ent}(f^2) \leqslant \mathrm{Ent}((f-\mu f)^2) + 2\mu (f-\mu f)^2, \label{eqPreLast2}
   \eeqn
which implies an inequality as $\eta \mathrm{Ent}(f^2) \leqslant B_1\mathcal{E}[f] + B_2\mu (f-\mu f)^2$. Combining it with the Poincar\'{e} inequality again gives
  \[ \eta \mathrm{Ent}(f^2) \leqslant C\mathcal{E}[f] \]
with $C=1+\eta(C_1+2C_4) +  \lambda_\mu^{-1}[A+\eta(2+C_2+C_3)]$.
\end{proof}

\bigskip

\section{Proof of Theorem \ref{thmLSILya}}
\label{LSIToLya}
 \setcounter{equation}{0}

The construction of Lyapunov function comes from solving an elliptic equation. For convenience, suppose a LSI holds as
  \[ 2\rho \mathrm{Ent}(f^2) \leqslant \mathcal{E}[f]. \]
According to the Herbst's argument in Aida-Masuda-Shigekawa \cite{AMSh}, a LSI implies the Gaussian integrability with some $c>0$ and $x_0\in E$
  \[ \mu e^{cd^2(x,x_0)} < \infty, \]
see also \cite[Proposition 5.4.1]{BGL} or \cite[Section 5.1]{Ledoux-01}.

Denote $\phi(x) = \rho\left(-cd^2(x,x_0) + b\right)$ with $b=2\mu e^{cd^2(x,x_0)}$. We introduce the equation for $f\in L^2(\mu)$
  \beqn
     \mathrm{H}u := -\mathrm{L}u + \phi u = f. \ \ \label{eqSchr}
  \eeqn
Our aim is to find a continuous positive solution $u$ to (\ref{eqSchr}) for $f\equiv 1$, which fulfills the definition of Lyapunov function and thus prove Theorem \ref{thmLSILya}.

\begin{proof} The strategy contains two steps.

{\bf Step 1}. Equation (\ref{eqSchr}) gives $\mu(u\cdot \mathrm{H}u) \leqslant \mathcal{E}[u] + \rho b \mu u^2$ quickly. On the other hand, using the Young's inequality and LSI yields
  \beq
     \mu(u\cdot \mathrm{H}u)
     &=& \mathcal{E}[u] + \rho b \mu u^2 - \rho \mu\left( cd^2(x,x_0)u^2 \right)\\
     &\geqslant& \mathcal{E}[u] + \rho b \mu u^2 - \rho\cdot\mu u^2 \cdot \mu\left( e^{cd^2(x,x_0)} - \frac{u^2}{\mu u^2} + \frac{u^2}{\mu u^2}\log\frac{u^2}{\mu u^2} \right)\\
     &=& \mathcal{E}[u] + \rho b \mu u^2 - \frac{b}{2}\rho\mu u^2 + \rho\mu u^2 - \rho\mathrm{Ent}(u^2) \ \geqslant\ \frac{1}{2}(\mathcal{E}[u] + \rho b \mu u^2).
  \eeq
Then $\mu(u\cdot \mathrm{H}u)$ determines a coercive Dirichlet form, and $\mathrm{H}$ is a positive definite self-adjoint Schr\"{o}dinger operator with its spectrum contained in $(0,\infty)$. It means $\mathrm{H}^{-1}$ exists on $L^2(\mu)$ according to the Lax-Milgram Theorem, i.e. $u=\mathrm{H}^{-1}f\in H^1(\mu)$ (i.e. the $L^2$-integrable Sobolev space of weak derivatives of first order) is a weak solution to Equation (\ref{eqSchr}).

Whenever $f\geqslant 0$, the weak maximum principle yields $u=\mathrm{H}^{-1}f \geqslant 0$ $\mu$-a.e. too. As a routine, we set $u_{-} = -\min\{u, 0\}$, which has weak derivatives and satisfies
   \[ \mu(u_{-}f) = \mu(u_{-}\cdot \mathrm{H}u) = - \mathcal{E}[u_{-}] - \mu(\phi u_{-}^2)
       \leqslant -\frac12 \left\{ \mathcal{E}[u_{-}] + \rho b\mu(u_{-}^2) \right\} \leqslant 0. \]
It follows $u_{-} = 0$ $\mu$-a.e. and thus $u\geqslant 0$ $\mu$-a.e.

{\bf Step 2}. Now fix $f\equiv 1$. By \cite[Theorem 8.22]{GiTr} and the notation therein, $u$ is locally H\"{o}lder continuous if we set $f^i=0$, $g=-f$ and $L = \mathrm{L} - \phi$ such that $Lu = g$.

Moreover, we prove that $u>0$ everywhere. By contradiction, assume $u(y)=0$ at some $y$. Choose $r>0$ and $w\in C^2(E)$ arbitrarily to satisfy
  \[ w\geqslant0, \ w(y)>0, \ w_{|B_r(y)^c} = 0. \]
Since $\mathrm{H}w$ is a bounded function, we can take a factor $\lambda>0$ and $v = \lambda w$ such that
  \[ \mathrm{H} v =  \lambda \mathrm{H} w \leqslant \frac12. \]
It follows on $E$
  \[ \mathrm{H}(u-v) = f - \mathrm{H}v \geqslant \frac12, \]
which implies $u-v\geqslant0$ $\mu$-a.e. by the weak maximum principle. Then $u(y)\geqslant v(y)$ by the continuity, which is absurd.

As consequence, we obtain $u\in H^1(\mu)\cap C(E)$ is strictly positive everywhere, and $u^{-1}$ is locally bounded. So $u$ is a Lyapunov function for $\mathrm{L}u \leqslant \phi u$.
\end{proof}

\bigskip

\subsection*{Acknowledgements}

{\small It is my great pleasure to thank Prof. Li-Ming Wu for some helpful conversations, and also thank the anonymous referee for his/her careful reading of the first version and pointing out some mistakes. This work is supported by NSFC (no. 11201456, and no. 1143000182), CAS (no. Y129161ZZ1), and Key Laboratory of Random Complex Structures and Data, Academy of Mathematics and Systems Science, Chinese Academy of Sciences (No. 2008DP173182).}

\bigskip

%\subsection*{Acknowledgements}

%{\small The authors sincerely thank}

% The above example implies that $P^n$ is asymptotic stable on $X$, but the unique ergodic support is $\{0\}$, which has no interior. So Condition $(\mathbf{A}1)$ is not necessary, however, the next example shows that it is sharp.

% BibTeX users please use
% \bibliographystyle{}
% \bibliography{}
%
% Non-BibTeX users please use

\end{document}